\documentclass[12pt,a4paper]{article}
\usepackage{amsfonts}
\usepackage{amsmath}
\usepackage{amsbsy}
\usepackage{amsxtra}
\usepackage{latexsym}
\usepackage{amssymb}
\usepackage{enumerate}
\makeatletter
\g@addto@macro\thesection.
\makeatother

\newtheorem{theorem}{Theorem}
\newtheorem{lemma}{Lemma}
\newtheorem{corollary}{Corollary}

\newtheorem{remark}{Remark}

\newtheorem{example}{Example}

\DeclareMathOperator{\cone}{cone}
\DeclareMathOperator{\inte}{int}

\DeclareMathOperator{\bdr}{bdr}

\newcommand{\lng}{\langle}
\newcommand{\rng}{\rangle}
\newcommand{\lf}{\left}
\newcommand{\rg}{\right}
\newcommand{\sa}{\sqcap}
\newcommand{\su}{\sqcup}
\newcommand{\R}{\mathbb R}
\newcommand{\N}{\mathbb N}

\newenvironment{proof}{{\noindent\bf Proof.}}{\hfill$\Box$\\}

\begin{document}

\title{Lattice-like operations and isotone projection sets\thanks{{\it 1991 A M S Subject Classification:} Primary 90C33,
Secondary 15A48; {\it Key words and phrases:} convex sublattices, isotone projections. }}
\author{A. B. N\'emeth\\Faculty of Mathematics and Computer Science\\Babe\c s Bolyai University, Str. Kog\u alniceanu nr. 1-3\\RO-400084 Cluj-Napoca, Romania\\email: nemab@math.ubbcluj.ro \and S. Z. N\'emeth\\School of Mathematics, The University of Birmingham\\The Watson Building, Edgbaston\\Birmingham B15 2TT, United Kingdom\\email: nemeths@for.mat.bham.ac.uk}
\date{}
\maketitle

\begin{abstract}

By using some lattice-like operations which constitute extensions of ones
introduced by
M. S. Gowda, R. Sznajder and J. Tao for self-dual cones, a new perspective
is  gained on the subject of isotonicity of the metric
projection onto the closed convex sets. The results of this paper are 
wide range generalizations of some results of the authors obtained for
self-dual cones. The aim of the subsequent investigations is
to put into evidence some closed convex sets for which the
metric projection is isotonic with respect the order relation
which give rise to the above mentioned lattice-like operations.
The topic is related to 
variational inequalities where the isotonicity of the metric projection is 
an important technical tool. For Euclidean sublattices this approach was considered 
by G. Isac and respectively by H. Nishimura and E. A. Ok.

\end{abstract}

\section{Introduction}

The idea of solving operatorial equations via iterative methods based on ordering goes back to 
the beginnings of the ordered vector space theory (see e. g. \cite{Krasnoselskii1962}).
If an operator has some ``good properties'' with respect to the ordering
of the space, then these can be exploited to derive iterative
processes for solving equations which involve such an operator.
The most important role in this context have
monotone or isotone operators, i. e., the operators
maintaining the order relation.
Iterative methods are widely used for 
solving various types of equilibrium problems (such as variational inequalities, 
complementarity problems etc.)
Recently the isotonicity gained more and more ground
for handling such problems (see \cite{NishimuraOk2012}, \cite{CarlHeikilla2011} and the 
large number of references in \cite{CarlHeikilla2011} related to ordered vector spaces). 

The pioneers of this approach for 
complementarity problems are G. Isac and A. B. N\'emeth. The metric projection onto a closed 
convex cone is the basis of many investigations about the solvability and/or the approximation
of solutions of nonlinear complementarity problems associated with the cone. However, the idea
of relating the ordering to the metric projection onto the cone is initiated in the paper
\cite{IsacNemeth1986}, where isotone projection cones (i.e., generating pointed closed convex 
cones admitting an isotone projection onto themselves) are characterised.
%

The papers \cite{IsacNemeth1990b}, \cite{IsacNemeth2008c},
\cite{Nemeth2009a} consider the problem of solving nonlinear complementarity problems defined 
on isotone projection cones. The solution methods require repeated projection onto the 
underlying cone. It turned out later that this procedure is efficient for isotone projection 
cones \cite{NemethNemeth2009}.


A similar approach can be considered for variational inequalities if the projection onto 
the closed convex set associated with the variational inequality is monotone with respect to 
an appropriate order relation. One of the simplest such order relation, the coordinatewise 
ordering, was considered by G. Isac \cite{Isac1996}, who proved that the sublatticiality of a 
closed convex set implies the isotonicity of the projection onto it.

Recently H. Nishimura and E. A. Ok followed the footsteps of G. Isac
and showed that the sublatticiality is also necessary for the the projection to be
isotone \cite{NishimuraOk2012}. They used the derived equivalence for several applications
concerning variational inequalities defined on closed convex sublattices and other related
equilibrium problems. Thus, the question of characterizing the closed convex 
sublattices with nonempty interior of the coordinatewise ordered Euclidean space which 
admit an isotone projection onto themselves with respect to this order arises very naturally.
A partial answer to this question can be found in the early papers of D. M. Topkis 
\cite{Topkis1976} and A. F. Veinott Jr. \cite{Veinott1981} and a complete one in the paper of 
M. Queyranne and  F. Tardella \cite{QueyranneTardella2006}.


The nonnegative orthant of a Cartesian system in the Euclidean space is a
self-dual latticial cone which is the positive cone of the coordinatewise ordering and 
defines ``well behaved'' lattice operations. Although important and therefore widely 
investigated, at the same time they are also very restrictive. There are several attempts to 
extend these 
lattice operations. One such extension proposed by M. S. Gowda, R. Sznajder 
and J. Tao \cite{GowdaSznajderTao2004} and related to self-dual cones is particularly
meaningful to us, because, apart from inheriting several properties of the lattice operations, 
their operations seem to be useful for handling the problem of the isotonicity of the metric
projections too.

In our recent paper \cite{NemethNemeth2012a} we have characterized the closed 
convex sets which are invariant with respect to these 
operations, showing that the metric projection onto these sets is isotone with respect to the 
order generated by the self-dual cone giving rise to the respective operations.

In this paper we show similar results for a general cone, by using two kinds of extended 
lattice operations: ones defined with the aid of the cone, while the others with the aid of its 
dual. These operations are strongly related: one can be expressed in terms of the other.
However, the parallel usage of them offers some techniqual facilities. Apart from extending 
the results of \cite{NemethNemeth2012a}, we show the equivalence between the sets which are 
invariant with respect to these operations and the ones which admit an isotone projection onto 
themselves.

The structure of this paper is as follows. In Section 2 
we fix the terminology regarding vectorial ordering. In Section 3
we define our main tools: the so called lattice-like operations, emphasizing the
relations with the lattice operations engendered by the nonnegative orthant and
the extended lattice operations defined by Gowda Sznajder and Tao as well.
In
the same section we gather and prove the main properties of these operations.
The main results and their proofs are contained in Section 4
namely in Theorems \ref{ISOINV} and \ref{FOOO}, and
Corollary \ref{POLYH}. They not only largely extend, but also strengthen the
results in \cite{NemethNemeth2012a}. As a consequence a full geometric
characterization of the closed convex sets admitting isotone
projection with respect to the ordering induced by the Lorentz cone is gained.
Besides the Lorentz cone, a special interest is focused on
the latticial or simplicial cones in Section 5%
.

Finally, we end our paper by
making some comments and raising some open questions in Section 6%
.

\section{Some terminology} \label{Some}
Denote by $\R^m$ the $m$-dimensional Euclidean space endowed with the standard inner 
product $\lng\cdot,\cdot\rng:\R^m\times\R^m\to\R.$ 

Throughout this note we shall use some standard terms and results from convex geometry 
(see e.g. \cite{Rockafellar1970}). 

Let $K$ be a \emph{convex cone} in $\R^m$, i. e., a nonempty set with
(i) $K+K\subset K$ and (ii) $tK\subset K,\;\forall \;t\in \R_+ =[0,+\infty)$. All cones used 
in this paper are convex.
The convex cone $K$ is called \emph{pointed}, if $(-K)\cap K=\{0\}.$

The cone $K$ is {\it generating} if $K-K=\R^m$.
 
For any $x,y\in \R^m$, by the equivalence $x\leq_K y\Leftrightarrow y-x\in K$, the 
convex cone $K$ 
induces an {\it order relation} $\leq_K$ in $\R^m$, that is, a binary relation, which is 
reflexive and transitive. This order relation is {\it translation invariant} 
in the sense that $x\leq_K y$ implies $x+z\leq_K y+z$ for all $z\in \R^m$, and 
{\it scale invariant} in the sense that $x\leq_Ky$ implies $tx\leq_K ty$ for any $t\in \R_+$.
If $\leq$ is a translation invariant and scale invariant order relation on $\R^m$, then 
$\leq=\leq_K$ with $K=\{x\in\R^m:0\leq x\}.$ The vector space $\R^m$
endowed with the relation $\leq_K$ is denoted by $(\R^m,K).$ If $K$ is pointed, then $\leq_K$ is 
\emph{antisymmetric} too, that is $x\leq_K y$ and $y\leq_K x$ imply that $x=y.$
The elements $x$ and $y$ are called \emph{comparable} if $x\leq_K y$
or $y\leq_K x.$

We say that $\leq_K$ is a \emph{latticial order} if for each pair
of elements $x,y\in \R^m$ there exist the least upper bound
$\sup\{x,y\}$ (denoted by $x\vee y$) and the greatest lower bound $\inf\{x,y\}$
(denoted by $x\wedge y$) of
the set $\{x,y\}$ with respect to the order relation $\leq_K$.
In this case $K$ is said to be a \emph{latticial or simplicial cone},
and $\R^m$ equipped with a latticial order is called 
\emph{Euclidean vector lattice}.

The \emph{dual} of the convex cone $K$ is the set
$$K^*:=\{y\in \R^n:\;\lng x,y\rng \geq 0,\;\forall \;x\in K\}.$$

The cone $K$ is called \emph{self-dual}, if $K=K^*.$ If $K$
is self-dual, then it is a generating, pointed, closed cone.

Suppose that $\R^m$ is endowed with a
Cartesian system. Let
$x,y\in \R^m$, $x=(x^1,...,x^m)$, $y=(y^1,...,y^m)$, where $x^i$, $y^i$ are the coordinates of
$x$ and $y$, respectively with respect to the Cartesian system. Then, the scalar product of $x$
and $y$ is the sum
$\lng x,y\rng =\sum_{i=1}^m x^iy^i.$

The set
\[\R^m_+=\{x=(x^1,...,x^m)\in \R^m:\; x^i\geq 0,\;i=1,...,m\}\]
is called the \emph{nonnegative orthant} of the above introduced Cartesian
system. A direct verification shows that $\R^m_+$ is a
self-dual cone.

Taking a Cartesian system in $\R^m$ and using the above introduced
notations,
the 
\emph{coordinatewise order}  $\leq$ in $\R^m$ is defined by
\[x=(x^1,...,x^m)\leq y=(y^1,...,y^m)\;\Leftrightarrow\;x^i\leq y^i,\;i=1,...,m.\]
By using the notion of the order relation induced by a cone, defined in the preceding
section, it is easy to see that $\leq =\leq_{\R^m_+}$.

With the above representation of $x$ and $y$, we define
$$x\wedge y=(\min \{x^1,y^1\},...,\min \{x^m,y^m\}),\;\;\textrm{and}\;\;x\vee y=(\max \{x^1,y^1\},...,\max \{x^m,y^m\}).$$

Then, $x\wedge y$ is  the greatest lower bound and $x\vee y$ is the least upper bound of 
the set $\{x,y\}$ with respect to the coordinatewise order. Thus, $\leq$ is a lattice order in $\R^m.$
The operations $\wedge$ and $\vee$ are called \emph{lattice operations}.

The subset $M\subset \R^m$ is called a \emph{sublattice of
the coordinatewise ordered Euclidean space} $\R^m$, if from
$x,y\in M$ it follows that $x\wedge y,\;x\vee y\in M.$ 

A \emph{hyperplane} (through the origin), is a set of form
\begin{equation}\label{hypersubspace}
H(u,0)=\{x\in \R^m:\;\lng u,x\rng =0\},\;\;u\not=0.
\end{equation}
For simplicity the hyperplanes will also be denoted by $H$.
The nonzero vector $u$ in the above formula is called \emph{the normal}
of the hyperplane. 

A \emph{hyperplane} (through $a\in\R^m$) is a set of form
\begin{equation}\label{hyperplane}
	H(u,a)=\{x\in \R^m:\;\lng u,x\rng =\lng u,a\rng\},\;\;u\not= 0.
\end{equation}

A hyperplane $H(u,a)$ determines two \emph{closed halfspaces} $H_-(a,u)$ and
$H_+(u,a)$  of $\R^m$, defined by

\[H_-(u,a)=\{x\in \R^m:\; \lng u,x\rng \leq \lng u,a\rng\},\]
and
\[H_+(u,a)=\{x\in \R^m:\; \lng u,x\rng \geq \lng u,a\rng\}.\]

\section{Metric projection and lattice-like operations} \label{Metric}

Denote by $P_D$ 
the projection mapping onto a nonempty closed convex set $D\subset \R^m,$ 
that is the mapping which associates
to $x\in \R^m$ the unique nearest point of $x$ in $D$ (\cite{Zarantonello1971}):

\[ P_Dx\in D,\;\; \textrm{and}\;\; \|x-P_Dx\|= \inf \{\|x-y\|: \;y\in D\}. \]

The nearest point $P_Dx$ can be characterized by

\begin{equation}\label{charac}
P_Dx\in D,\;\;\textrm{and}\;\;\lng P_Dx -x,P_Dx-y\rng \leq 0 ,\;\forall y\in D.
\end{equation}

From the definition of the projection or the characterization (\ref{charac}) there follow 
immediately the relations: 

%
\begin{equation}\label{et}
	P_{x+D}y=x+P_D(y-x) 
\end{equation}
for any $x,y\in\R^m$,

\begin{equation}\label{sun}
P_D(tx+(1-t)P_Dx)=P_Dx,
\end{equation}
for any $x,y\in\R^m$ and any $t\in [0,1]$.

We shall frequently use in the sequel the following simplified form of the theorem of Moreau \cite{Moreau1962}:

\begin{lemma}\label{lm}
        Let $K$ be a closed convex cone in $\R^m$ and $K^*$ its dual. For any $x$ in $\R^m$ we 
	have $x=P_Kx-P_{K^*}(-x)$ and $\lng P_Kx,P_{K^*}(-x)\rng=0$. The relation $P_Kx=0$ 
	holds if and only if $x\in -K^*$.
\end{lemma}

Define the following operations in $\R^m$: 
\[x\sa y=P_{x-K}y,\textrm{ }x\su y=P_{x+K}y,\textrm{ }x\sa_* y=P_{x-K^*}y,\textrm{ and }x\su_* y=P_{x+K^*}y\]
 Assume that
the operations $\su$, $\sa$, $\su_*$ and $\sa_*$ have precedence over the addition of vectors and 
multiplication of vectors by scalars.

If $K$ is self-dual, then $\su = \su_*$ and $\sa =\sa_*$ and we arrive to the generalized
lattice operations defined by Gowda, Sznajder and Tao in \cite{GowdaSznajderTao2004}, and used 
in our paper \cite{NemethNemeth2012}. In the particular case of 
$K=\R^m_+$, one can easily check that $\sa=\sa_*=\wedge,\; \su=\su_*=\vee.$
That is $\sa$, $\sa_*$, $\su$ and $\su_*$ are some \emph{lattice-like operations}.

Although derived from the metric projection, the generalized lattice operations due to
Gowda, Sznajder and Tao as well as the above considered lattice-like operations
introduce a wieldy formalism, which allows the recognition of new interrelations in the field. 

Lemma \ref{lm} suggests strong connections between the lattice-like operations. These 
connections are exhibited by the following lemma:

\begin{lemma}\label{l0}
	The following relations hold for any $x,y\in\R^m$ : 
	\begin{enumerate}
		\item[\emph{(i)}] \[x\sa y=x-P_K(x-y)=y-P_{K^*}(y-x),\] 
			\[x\su y=x+P_K(y-x)=y+P_{K^*}(x-y),\] \[x\sa_*y=x-P_{K^*}(x-y)=y-P_K(y-x),\] 
			\[x\su_*y=x+P_{K^*}(y-x)=y+P_K(x-y).\] 
		\item[\emph{(ii)}] $x\sa_*y=y\sa x$  and $x\su_*y=y\su x$.
		\end{enumerate}
	\end{lemma}	

\begin{proof}
	\begin{enumerate}
		\item[(i)] From equation (\ref{et}) and Lemma \ref{lm} we have 
			\[
			\begin{array}{rcl}
				x\su y & = & P_{x+K}y=x+P_K(y-x)=
				x+(P_K(y-x)-P_{K^*}(x-y))+P_{K^*}(x-y)\\
				& = & x+(y-x)+P_{K^*}(x-y)=y+P_{K^*}(x-y).
			\end{array}
			\]
			The other relations can be shown similarly.
		\item[(ii)] It follows easily from item (i).
		\end{enumerate}
		\end{proof}

Denote by $\le$ and $\le_*$ the relations defined by $K$ and $K^*$, respectively.

\begin{lemma}\label{ll}
	The following relations hold for any $x,y,z,w\in\R^m$ and any real scalar 
	$\lambda>0$. 
	\begin{enumerate}
		\item[\emph{(i)}] $x\sa y\le x$, $x\sa y\le_*y$, $x\sa_*y\le_*x$ and $x\sa_*y\le y$, and equalities hold if and 
			only if $x\le_*y$, $y\le x$, $x\le y$ and $y\le_*x$, respectively.
		\item[\emph{(ii)}] $x\le x\su y$, $y\le_*x\su y$, $x\le_*x\su_*y$ and $y\le x\su_*y$, and equalities hold if and 
			only if $y\le_*x$, $x\le y$, $y\le x$, $x\le_*y$, respectively.
		\item[\emph{(iii)}] $x\sa y+x\su_*y=x\sa_*y+x\su y=x\sa y+y\su x=x\sa_*y+y\su_* x=x+y$. 
		\item[\emph{(iv)}] $(x+z)\sa (y+z)=x\sa y+z$, $(x+z)\su (y+z)=x\su y+z$. $(x+z)\sa_*(y+z)=x\sa_*y+z$ and
			$(x+z)\su_*(y+z)=x\su_*y+z$.
		\item[\emph{(v)}] $(\lambda x)\sa (\lambda y)=\lambda x\sa y$,
			$(\lambda x)\su (\lambda y)=\lambda x\su y$, $(\lambda x)\sa_*(\lambda y)=\lambda x\sa_*y$ and 
			$(\lambda x)\su_*(\lambda y)=\lambda x\su_*y$.
		\item[\emph{(vi)}] $\lng x-x\sa y,x\su_*y-x\rng=0$ and $\lng x-x\sa_*y,x\su y-x\rng=0$. 
		\item[\emph{(vii)}] $(-x)\su(-y)=-x\sa y$ and $(-x)\su_*(-y)=-x\sa_*y$.
		\item[\emph{(viii)}] $\|x\sa y-z\sa w\|\leq \frac{3}{2}(\|x-z\|+\|y-w\|)$ and
		$\|x\su y-z\su w\|\leq \frac{3}{2}(\|x-z\|+\|y-w\|)$.
		\item[\emph{(ix)}]
		\[x\sa y=z\sa w,\;\;\forall \;\;z=\lambda x+(1-\lambda)x\sa y,\;\;w=\mu y+(1-\mu)x\sa y,\;\,\lambda,\;\mu \in [0,1],\]
	  \[x\su y=z\su w,\;\;\forall \;\;z=\lambda x+(1-\lambda)x\su y,\;\;w=\mu y+(1-\mu)x\su y,\;\,\lambda,\;\mu \in [0,1].\]
	\end{enumerate}
\end{lemma}

\begin{proof}
\begin{enumerate}

		\item[(i)] It follows from the definitions of the operations, Lemma \ref{l0} and Lemma \ref{lm}. 
		\item[(ii)] It can be shown similarly to item (i).
\end{enumerate}

	Items (iii) and (iv) follow immediately from item (i) and (ii) of Lemma \ref{l0}. 
	
	Item (v) 
	follows easily from the positive homogeneity of $P_K$ and $P_{K^*}$, and item (i) of Lemma \ref{l0}. 
\begin{enumerate}
	\item[(vi)] By using item (i) of Lemma \ref{l0} and Lemma \ref{lm}, we get
			\[\lng x-x\sa y,x\su_*y-x\rng=\lng P_K(x-y),P_{K^*}(y-x)\rng=0.\]
\end{enumerate}
	
	Item (vii) follows from 
	item (i) of Lemma \ref{l0}.
	
	To verify item (viii) we use item (i) of Lemma \ref{l0} and the Lipschitz property of the
	metric projection (\cite{Zarantonello1971}) as follows:
	\[\|x\sa y-z\sa w\| =\|x-P_K(x-y)-z+P_K(z-w)\|\leq \|x-z\|+\|P_K(x-y)-P_K(z-w)\|\leq \]
	\[\|x-z\|+\|(x-y)-(z-w)\|\leq 2\|x-z\|+\|y-w\|,\]
	and by symmetry
	\[\|x\sa y-z\sa w\| \leq \|x-z\|+2\|y-w\|.\]
	By adding the obtained two relations we conclude the first relation in item (viii).
	The second relation can be deduced similarly. 
	
	By using the definition of $x\sa y$, we have, according to the formula (\ref{sun}), that
	\[x\sa y=P_{x-K}y=P_{x-K}(\mu y+(1-\mu)x\sa y)=P_{x-K}w =x\sa w =w\sa_*x.\]
	Now, according to this formula, the formula
	$z=\lambda x+(1-\lambda)w\sa_*x=\lambda x +(1-\lambda)P_{w-K^*}x$, and by swapping the 
	roles of $K$ and $K^*$ and using a similar argument as above, 
	we obtain
	\[w\sa_*x=P_{w-K^*}x=P_{w-K^*}(\lambda x+(1-\lambda)P_{w-K^*}x)=P_{w-K^*}z=
	w\sa_*z=z\sa w.\] In conclusion,
	\[x\sa y=z\sa w.\]
	This is the first formula in item (ix). 
	A similar argument yields the second relation in this item.
	%
	\end{proof}

\begin{remark}

Lemma \ref{l0} and Lemma \ref{ll} show that the operations $\sa_*$ and $\su_*$
can be expressed everywhere in what follows by $\sa$ and $\su$.
However, we shall occasionally use the first ones too
in order to emphasize certain assertions and to simplify the arguments.

\end{remark}

\section{Closed convex sets invariant with respect to the
lattice-like operations} 
\label{Closed}

The set $M\subset \R^m$ is said to be \emph{invariant with respect to the
operation} $\sa $ if from $x,\;y\in M$ it
follows that $x\sa y\in M$.  The
invariance of $M$ with respect to any of the operations $\su$, $\sa_*,$  and $\su_*$ can
be defined similarly.

The following lemma follows easily from item (ii) of Lemma \ref{l0}

\begin{lemma}\label{linv}
	Let $K\subset \R^m$ be a closed convex cone. If $C$ is invariant with respect to 
	one of the operations $\sa$, $\sa_*$ and one of the operations $\su$, $\su_*$, then 
	$C$ is invariant with respect to all operations $\sa$, $\sa_*$, $\su$ and $\su_*$.
\end{lemma}

Let $K$ be a nonzero closed convex cone. We should simply call a set $M$ which is invariant 
with respect to the operations $\sa$, $\sa_*$, $\su$ and $\su_*$ \emph{$K$-invariant}. By Lemma 
\ref{linv}, it is enough to suppose that $M$ is invariant with respect to $\sa$ and $\su$, or 
$\sa_*$ and $\su$, or $\sa$ and $\su_*$, or $\sa_*$ and $\su_*$.

Recall that $\le$ and $\le_*$ denote the relations defined by $K$ and $K^*$, respectively.

If $K$ is a nonzero closed convex cone, then the closed convex set $C\subset \R^m$ is called a
\emph{$K$-isotone ($K^*$-isotone) projection set} or simply \emph{$K$-isotone ($K^*$-isotone)}
if $x\leq y$ implies $P_C x\leq P_C y$ (and respectively
$x\leq_* y$ implies $P_C x\leq_* P_C y$). In this case we use equivalently
the term \emph{$P_C$ is $K$-isotone} (respectively \emph{$P_C$ is $K^*$-isotone}).

\begin{theorem}\label{ISOINV}
	Let $K\subset \R^m$ be a closed convex cone. Then, $C$ is $K$-invariant, if and only if $P_C$ is 
	$K$-isotone. 
\end{theorem}

\begin{proof}
Assume that the closed convex set $C$ is $K$-invariant.
	Let $x,y\in \R^m$ with $x\le y$ and denote $u=P_Cx\in C$, $v=P_Cy\in C$. 
	
	Assume that $u\le v$ is false. Then, from $u\su v\in C$, the definition of the
	projection and item (ii) of Lemma \ref{ll}, we have $\|y-v\|<\|y-u\su v\|$. Hence,
	from
	\[\|y-v\|^2=\|y-u\su v\|^2+\|u\su v-v\|^2+2\lng y-u\su v,u\su v-v\rng,\] 
	it follows that 
	\[\|u\su v-v\|^2<2\lng u\su v-y,u\su v-v\rng.\]
	On the other hand, since $u\sa_*v\in C$, we 
	have $\|x-u\|\le\|x-u\sa_*v\|$, and thus we have similarly that
	\[\|u\sa_*v-u\|^2\le2\lng u\sa_*v-x,u\sa_*v-u\rng.\]
	By summing up the latter two inequalities and using item (iii) of Lemma \ref{ll}, it 
	follows that
	\[\lng u\su v-v,u\su v-v\rng=\|u\su v-v\|^2<\lng u\su v-y,u\su v-v\rng+
	\lng x-u\sa_*v,u\su v-v\rng.\]
	Thus, 
	\[\lng y-x-(v-u\sa_*v),u\su v-v\rng<0.\]
	By combining the latter inequality with item (vi) of Lemma \ref{ll}, we obtain that 
	\[\lng y-x,u\su v-v\rng<0.\] But this is a contradiction, because $y-x\in K$ and
	$u\su v-v\in K^*$ (by item (ii) of Lemma \ref{ll}).
	
	The obtained contradiction shows that $P_C$ must be $K$-isotone.

Let us see now that if $P_C$ is $K$ isotone, then $C$ is $K$-invariant.

Assume the contrary: $P_C$ is $K$-isotone, but there exist $x,\;y\in C$ such that either
$x\sa y\notin C$, or $x\su y\notin C$.

Assume, that $x\sa y\notin C$.

Since $x\sa y \leq x$ and $P_C$ is $K$-isotone, it follows that $P_C(x\sa y)\leq x,$
that is, $P_C(x\sa y)\in x-K.$  By our working hypothesis $x\sa y\not= P_C(x\sa y)$ and
by the definition of $x\sa y= P_{x-K}y$, we must have
\begin{equation}\label{ellen1}
\|y-x\sa y\|<\|y-P_C(x\sa y)\|.
\end{equation}

Since $y\in C$, by the characterization \eqref{charac} of the projection we have:
\begin{gather*}
	\lng y-P_C(x\sa y),x\sa y-y\rng+\|y-P_C(x\sa y)\|^2\\
	=\lng y-P_C(x\sa y),x\sa y-y+y-P_C(x\sa y)\rng\\
	=\lng y-P_C(x\sa y),x\sa y-P_C(x\sa y)\rng\le0,
\end{gather*}

which, by using the Cauchy inequality, implies 
\begin{equation}\label{edouble}
	\|y-P_C(x\sa y)\|^2\le\lng y-P_C(x\sa y),y-x\sa y\rng\le\|y-P_C(x\sa y)\|\|y-x\sa y\|,
\end{equation}
If $y=P_C(x\sa y)$, then the inequality
\begin{equation}\label{ejav}
	\|y-P_C(x\sa y)\|\le\|y-x\sa y\|,
\end{equation}
holds trivially. If $y\neq P_C(x\sa y)$, then \eqref{ejav} follows from dividing 
\eqref{edouble} by $\|y-P_C(x\sa y)\|$. However, \eqref{ejav} contradicts (\ref{ellen1}).

The case of $x\su y\notin C$ can be handled similarly.

The obtained contradictions show that $C$ must be $K$-invariant. 
\end{proof}

\begin{example}
The set
\[L_{m+1}=\{(x,x^{m+1})\in\R^{m+1}:\textrm{ }x\in\R^m,\textrm{ }x^{m+1}\in\R\textrm{ and }
\|x\|\leq x^{m+1}\},\] 
 is a self-dual cone called 
\emph{$m+1$-dimensional second order cone}, or 
\emph{$m+1$-dimensional Lorentz cone}, or \emph{$m+1$-dimensional ice-cream cone} 
(\cite{GowdaSznajderTao2004}). 

By using Theorem \ref{ISOINV}, we can strengthen the main result in 
\cite{NemethNemeth2012a} regarding the Lorentz cone $L_{m+1}$ as follows:

Let $M$ be a closed convex subset with nonempty interior in $\R^{m+1}$ with $m>1$.
Then, the following assertions are equivalent:
\begin{enumerate}
\item[\emph{(i)}]  $M$ is invariant
with respect to the operations $\sa$ and $\su$ defined by the Lorentz
cone $L_{m+1}$,
\item[\emph{(ii)}] $M$ is an $L_{m+1}$-isotone projection set,
\item[\emph{(iii)}]
\begin{equation*} 
M=C\times \R, 
\end{equation*}
where $C$ is a closed convex set with nonempty interior in $\R^m$.
\end{enumerate}

In \cite{NemethNemeth2012a} we proved that (iii)$\Leftrightarrow$(ii)$\Rightarrow$(i), but not
the implication (i)$\Rightarrow$(ii).
\end{example}

\begin{lemma}\label{latprop}
	Let $K\subset \R^m$ be a closed convex cone. If 
	$M,\;M_i,\;\subset\R^m,\;\;i\in \mathcal I$ are $K$-invariant sets, then
	\begin{enumerate}
		\item[\emph{(i)}] $\cap_{i\in \mathcal I} M_i$ is also $K$-invariant,
		\item[\emph{(ii)}] $\eta M+a$ is also $K$-invariant for any $a\in\R^m$ and 
			$\eta\in\R$.
		
	\end{enumerate}
\end{lemma}

\begin{proof}
	The first assertion is trivial and the second follows easily from items (iv), (v) and 
	(vii) of Lemma \ref{ll}.
	\end{proof}

\begin{lemma}\label{halfhyper}
The halfspace $H_-$ is $K$-invariant 
if and only if the hyperplane $H$ has this property.
\end{lemma}
\begin{proof}
According to Lemma \ref{linv}, it is enough to carry out the proof for the
case of the invariance with respect to $\sa$ and $\su$.

According to item (ii) of Lemma \ref{latprop}, we can assume that $0\in H$.

Suppose that $H$ is invariant, but $H_-$ is not. 
Then, there exist some $x,y\in H_-$
such that $x\su_*y=y\su x\notin H_-$ or $x\sa y\notin H_-$. Assume that $x\sa y\notin H_-$.
Then, $x\sa y\in \inte H_+.$ The line segment $[x,x\sa y]$ meets $H$ in
$z=\lambda x +(1-\lambda)x\sa y,\;\lambda \in (0,1],$ while the line segment
$[y,x\sa y]$ meets $H$ at $w=\mu y  +(1-\mu)x\sa y,\;\mu \in (0,1].$
According to item (ix) in Lemma \ref{ll}, we have then
\[z\sa w=x\sa y\notin H,\]
which contradicts the invariance of $H$.

Suppose now that $H_-$ is invariant, but $H$ is not. Then, there exist some $x,y\in H$
such that $x\su_* y=y\su x\notin H$ or $x\sa y\not \in H$. Since $H_-$ is invariant, we can
assume that $x\su y\in \inte H_-$. Let $u$ be the normal of $H$. Then, $\lng u,x\su y\rng<0.$
By using the relation in item (iii) of Lemma \ref{ll}, we have then
\[0=\lng u,x+y\rng=\lng u, x\su y\rng +\lng u,y\sa x\rng.\]
Whereby, by using the relation $\lng u,x\su y\rng <0,$
we conclude that
\[\lng u,y\sa x\rng >0,\]
that is, $y\sa x\in \inte H_+$, contradicting the invariance of $H_-$. Similarly 
$x\sa y\notin H$ leads to a contradiction.
\end{proof}

\begin{lemma}\label{tanginvar}
Suppose that $C$ is a $K$-invariant closed convex set with nonempty interior,
and $H$ is a hyperplane tangent to $C$ in some point of $\bdr C$.
Then, $H$ is $K$-invariant.
\end{lemma}
\begin{proof}
According to item (ii) of Lemma \ref{latprop}, we can assume that $0\in \bdr C$, that
$H$ is tangent to $C$ at $0$, and that $C\subset H_-$.

We shall prove our claim by contradiction: we assume that $H$ is not
invariant.

Since $H$ is not invariant, there exist some $z,\;w\in H$ such that
$z\sa w$ or $w\su z$ is not in $H$. Suppose that $u$ is the normal
of $H$. From the relation in item (iii) of Lemma \ref{ll} we have then
\[0=\lng u,z+w\rng= \lng u,z\su w \rng+\lng u,w\sa z \rng,\]
whereby it follows that $z\su w$ and $w\sa z$ are in opposite open
half-spaces determined by $H$.

Suppose that $w\sa z\in \inte H_+.$ By taking $x=z-(z+w)/2,$
we have $-x=w-(z+w)/2.$ Then, by our working hypothesis that $0\in H,$
it follows that the line segment $[-x,x]\subset H.$ We can easily check 
that $(-x)\sa x\in \inte H_+.$ Denoting by $B$ the unit ball in $\R^m$,
then there exists some $\delta >0$ such that
\begin{equation}\label{gomb}
(-x)\sa x+\delta B\subset \inte H_+.
\end{equation}

We have the relation
\[[-x,x]=\{tx:\;t\in[-1,1]\}.\]
Next we project $[-x,x]$ in the direction of $u$ onto $\bdr C$.
All the above reasonings are valid when we change $x$
with its positive multiple, hence we can chose $x$ small enough,
so that the above projection to make a sense.

Denote by $\gamma (t)$ the image of $tx$ in $\bdr C$ by  this projection.
Since $H$ is a tangent hyperplane, the segment $[-x,x]$ will be tangent to
$\gamma $ at $t=0,\;\gamma (0)=0,$ $\gamma'(0)$ exists, and $\gamma'(0)=x.$

Since $\gamma$ is differentiable in $t=0$, we have the following
representations around $0$:
\begin{equation}\label{post}
\gamma (t)=tx+\eta (t),\;t>0,
\end{equation}
and
\begin{equation}\label{negt}
\gamma (-t)=-tx+\zeta (-t),\;t>0,
\end{equation}
where
\begin{equation}\label{kiso}
\frac{\eta (t)}{t}\to 0\;\;\textrm{and}\;\; \frac{\zeta (-t)}{t}\to 0,\;\;\textrm{as}\;t\to 0,\;t>0.
\end{equation}

Using item (viii) of Lemma \ref{ll}, as well as the relations (\ref{post}) and (\ref{negt}), we have then
\[\|(-tx)\sa (tx)- \gamma (-t)\sa \gamma (t)\|\leq \frac{3}{2}(\|-tx-\gamma(-t)\|+\|tx-\gamma (t)\|)=
\frac{3}{2}(\|\zeta (-t)\|+\|\eta (t)\|).\]

Dividing the last relation by $t>0$,  and using the relation in item (v) of
Lemma \ref{ll}, we obtain that
\begin{equation}\label{becsles}
\|(-x)\sa x-\frac{1}{t} \gamma(-t)\sa \gamma (t)\|\leq \frac{3}{2}\lf(\lf\|
\frac{\zeta(-t)}{t}\rg\|+\lf\|\frac{\eta (t)}{t}\rg\|\rg).
\end{equation}

Take now $t>0$ small enough in order to have by (\ref{kiso})
\[\frac{3}{2}\lf(\lf\|\frac{\zeta(-t)}{t}\rg\|+\lf\|\frac{\eta (t)}{t}\rg\|\rg)<\delta.\]
For such a $t>0$ we have, by using (\ref{becsles}), that
\[\frac{1}{t}(\gamma (-t)\sa \gamma (t))\in \inte H_+,\]
and thus
\[\gamma (-t)\sa \gamma (t) \in \inte H_+,\]
that is, $\gamma (-t),\;\gamma (t)\in C$, but
\[\gamma (-t)\sa \gamma (t)\notin C,\]
contradicting the invariance of $C$.

The obtained contradiction shows that $H$ must be invariant
with respect to the operations $\su$ and $\sa$, that is, by Lemma \ref{linv},
$H$ must be $K$-invariant.

\end{proof}

\begin{theorem}\label{FOOO}
The closed convex set $C\subset\R^m$ with nonempty interior is $K$-invariant
if and only if it is of the form
\begin{equation*}
C=\cap_{i\in \N} H_-(u_i,a_i),
\end{equation*}
where each hyperplane $H(u_i,a_i)$ is tangent to $C$ and is $K$-invariant.
\end{theorem}

\begin{proof}
It is known (see e.g. \cite{Rockafellar1970}, Theorem 25.5) that if $C\subset \R^m$ is a closed convex set
with nonempty interior, then $\bdr C$ contains a dense subset
of points where this surface is differentiable. Since the topology
of $\bdr C$ possesses a countable basis, we can select from this dense set a countable
dense set $\{a_i:\;i\in \N\}\subset \bdr C$ such that
there exist the tangent hyperplanes $H(u_i,a_i)$ to $C$ and
$C\subset H_-(u_i,a_i),\; i\in \N.$ Since the set $\{a_i,\;i\in \N\}$ is dense
in $\bdr C$, a standard convex geometric reasoning shows that in fact
\begin{equation}\label{vegsob}
C=\cap_{i\in \N} H_-(u_i,a_i).
\end{equation}
  
Now, if $C$ is $K$-invariant, then so is $H(u_i,a_i),\;i\in \N$	
by Lemma \ref{tanginvar}.
Hence, the necessity of the condition in the theorem is proved.

Conversely, if we have the representation (\ref{vegsob}) with the
hyperplanes $H(u_i,a_i),\;i\in \N$ $K$-invariant, 
 then, by Lemma \ref{halfhyper}, the halfspaces $H_-(u_i,a_i),\;i\in \N$ are also $K$-invariant.
Then, by using item (i) of Lemma \ref{latprop} and the representation (\ref{vegsob}), we see 
that $C$ is $K$-invariant, and the sufficiency of
the theorem is proved.

\end{proof}

By gathering and comparing the results of Theorem \ref{ISOINV}  
and Theorem \ref{FOOO}, we obtain
the following corollary:

\begin{corollary}\label{POLYH}
Let $C$ be a closed convex set with nonempty interior.
Then, the following assertions are equivalent:
 \begin{enumerate}
\item [\emph{(i)}] $C$ is a $K$-invariant set, i. e., it is
invariant with respect to the operations $\sa$, $\su$, $\sa_*,$ and $\su_*$ ;
\item [\emph{(ii)}] The projection $P_C$ is $K$-isotone and $K^*$-isotone;
\item [\emph{(iii)}] The set $C$ can be represented by
\begin{equation}\label{redu}
C=\cap_{i\in \N} H_-(u_i,a_i),
\end{equation}
where each hyperplane $H(u_i,a_i),\;i\in \N$ is tangent to $C$ and is $K$-invariant;
\item [\emph{(iv)}] $C$ can be represented by (\ref{redu}),
where each hyperplane $H(u_i,a_i),\;i\in \N$ is tangent to $C$ and is a 
$K$-isotone projection and a $K^*$-isotone
projection set.
\end{enumerate}
\end{corollary}

\begin{proof}

From Theorem \ref{FOOO} we have the equivalence 
\[\textrm{(i)} \Leftrightarrow \textrm{(iii)}.\]

From Theorem \ref{ISOINV} we have the equivalence
\[\textrm{(i)}\Leftrightarrow \textrm{(ii)}.\]

We have by Theorem \ref{ISOINV} that a hyperplane $H(u_i,a_i) (i\in \N)$ is
$K$-isotone if and only if it is invariant. Hence, we have the equivalence
\[\textrm{(iii)}\Leftrightarrow \textrm{(iv)}.\]

\end{proof}
\section{The case of the simplicial cone}\label{The}

If the closed convex cone $K\subset \R^m$ induces a latticial ordering $\leq$ in
$\R^m$, then it is called \emph{simplicial}. The origin of this term relies
in Youdine's theorem \cite{Youdine1939}, which says that in this case
\begin{equation}\label{youdine}
K=\cone \{e_1,...,e_m\}=\{t^1e_1+...+t^me_m,\;t^i\in \R_+,\;i=1,...,m\},
\end{equation}
where the vectors $e_1,...,e_m$ form a basis of $\R^m$.

If $K$ is a simplicial cone, the set $M\subset \R^m$ is called a \emph{sublattice of}
$(\R^m,K)$ if from $x,\;y\in M$ it follows that $x\vee y,\;x\wedge y\in M.$

The cone $K\subset \R^m$ is called \emph{subdual} if $K\subset K^*$.

The following lemma does not use the representation (\ref{youdine}) of
a simplicial cone $K$,  but only the latticiality of $\leq$. Again we put $\le$ for $\le_K$.

\begin{lemma}\label{NishOk}
	If the closed convex set $C\subset \R^m$ admits a $K$-isotone projection $P_C$ with
	respect to a subdual simplicial cone $K\subset\R^m$, then $C$ is a sublattice 
	of the lattice $(\R^m,K)$. 
\end{lemma}

\begin{proof}
	Suppose that $P_C$ is $K$-isotone and take $x,y\in C$.
	Let us see that $x\vee y\in C.$ 

	From the characterization (\ref{charac}) of the projection we have
	\begin{equation}\label{NishOkf}
		\lng P_C(x\vee y)-x\vee y,P_C(x\vee y)-y\rng \leq 0.
	\end{equation}
	Since $x\leq x\vee y$ and $P_C$ is $K$-isotone, it follows that $x=P_Cx\leq P_C(x\vee y)$. 
	Similarly, $y\leq P_C(x\vee y)$ and hence $x\vee y\leq P_C(x\vee y).$ We have also
	\begin{equation}\label{NO}
		0\leq P_C(x\vee y)-x\vee y\leq P_C(x\vee y)-y. 
	\end{equation}
	Thus, the two terms in the scalar product (\ref{NishOkf}) are in $K$
	and since $K$ is subdual, we must have the equality:
\begin{equation}\label{NishOkff}
\lng P_C(x\vee y)-x\vee y,P_C(x\vee y)-y\rng =0.
\end{equation}
 
By using again the subduality of $K$, the relation (\ref{NO}), as well as (\ref{NishOkff}), 
it follows that
\[0\leq \lng P_C(x\vee y)-x\vee y, (P_C(x\vee y)-y)- (P_C(x\vee y)-x\vee y)\rng=
-\|P_C(x\vee y)-x\vee y\|^2,\]
thus we must have
\[P_C(x\vee y)=x\vee y,\]
and since $C$ is closed, $x\vee y\in C.$

Similar reasonings show that $x\wedge y\in C.$
\end{proof}

The next corollary follows from Theorem \ref{ISOINV} and Lemma \ref{NishOk}.

\begin{corollary}\label{cis}
	Let $K\subset\R^m$ be a subdual simplicial cone. If the closed convex set 
	$C\subset \R^m$ is $K$-invariant, then $C$ is a sublattice of the lattice $(\R^m,K)$.
\end{corollary}

\begin{remark}
	Since the $K$-invariance of a set is equivalent to its $K^*$-invariance, by replacing 
	$K$ with $K^*$, similar results to Lemma \ref{NishOk} and Corollary \ref{cis} hold 
	for superdual cones too.
\end{remark}

It is well known (see e. g. \cite{NemethNemeth2009}) that if the simplicial cone $K$ is represented by (\ref{youdine}),
then we have for its dual the representation
\begin{equation}\label{youdinedual}
K^*=\cone \{u_1,...,u_m\}=\{t^1u_1+...+t^mu_m\;;\;t^i\in \R_+,\;i=1,...,m\},
\end{equation}
where the vectors $u_1,...,u_m$ are obtained from the relations
\begin{equation}\label{biort}
\lng e_i,u_j\rng=\delta_{ij},\; \textrm{for any}\;i,j\in\{1,\dots,m\}, 
\end{equation}
where $\delta_{ij}$ is the Kronecker delta.

\begin{lemma}\label{foo}
	Let $K\subset\R^m$ be a simplicial cone given by (\ref{youdine}) such that its dual $K^*$
	is of form (\ref{youdinedual}) with the vectors $u_1,\dots,u_m$ satisfying (\ref{biort}).
	The hyperplane $H$ through $0$ with the unit 
	normal vector $a$ is $K$-invariant if and only if 
	\begin{equation}\label{ebio}
		\lng a,e_i\rng\lng a,u_j\rng\le\delta_{ij},
	\end{equation} 
	for any $i,j\in\{1,\dots,m\}$. If $a=\alpha^1e_1+\dots+\alpha^me_m=\beta^1u_1+\dots+
	\beta^mu_m$, then $\alpha^i=\lng a,u_i\rng$ and $\beta^j=\lng a,e_j\rng$, and hence
	the system of inequalities \eqref{ebio} is equivalent to
	\begin{equation}\label{ebio2}
		\alpha^i\beta^j\le\delta_{ij},
	\end{equation}
\end{lemma}
\begin{proof}

	By Theorem \ref{ISOINV} it is enough to prove that $P_H$ is isotone if and only if
	the conditions of the lemma hold. 

	Since $P_H$ is linear, in order to characterize the hyperplane
	$H$ with the property that $x\leq y$ implies $P_Hx\leq P_Hy$, it is
	sufficient to give necessary and sufficient conditions on
	the unit vector $a$ such that
	\begin{equation}\label{charH}
		0\le P_H e_i,\;\;i=1,...,m.
	\end{equation}

	Since $a$ is a unit vector, the conditions (\ref{charH})
	can be written in the form:
	\begin{equation}\label{charu}
		0\le P_He_i =e_i-\lng a,e_i\rng a
		\;i=1,...,m.
	\end{equation}
These conditions are equivalent to
\begin{equation}\label{nonpositive}
	\lng e_i-\lng a,e_i\rng a,u_j\rng\ge 0,
\end{equation}
for all $i,j\in\{1,\dots,m\}$, which is equivalent to \eqref{ebio}.
The relations $\alpha^i=\lng a,u_i\rng$ and $\beta^j=\lng a,e_j\rng$ follow easily from 
formulas \eqref{biort}.

\end{proof}

\begin{corollary}\label{orthalo}
Let $e_1,...,e_m$ be an orthonormal system of vectors in $\R^m$ and consider
the system engendered by it in $\R^m$. Then, the hyperplane $H$ through $0$
with the unit normal $a=(a^1,...,a^m)$ is $K=\R_m^+$-invariant if and only if
and only if
\begin{equation}\label{parteset}
a^ia^j\leq 0 \;\textrm{whenever}\; i\not= j.
\end{equation}
\end{corollary}

\begin{proof}
Using the notation in Lemma \ref{foo}, we have in this case $u_i=e_i,\;i=1,...,m$,
and $\alpha^i=\beta^i=a^i,\;i=1,...,m$. From the condition $\|a\|=1$ we have
that $|a^i|\leq 1,\;i=1,...,m$ and hence the conditions which corresponds to
(\ref{ebio2}) for $i=j$, that is, $a^ia^i\leq 1,\;i=1,...,m$ are automatically satisfied. The remaining
conditions are exactly those in (\ref{parteset}).
\end{proof}

This corollary is in fact nothing else as Lemma 14 in \cite{NemethNemeth2012a}.

\begin{corollary}\label{cfoo}
	Let $K\subset\R^m$ be a simplicial cone given by (\ref{youdine}) such that its dual $K^*$
	is of form (\ref{youdinedual}) with the vectors $u_1,\dots,u_m$ satisfying (\ref{biort}).
	The existence of a $K$-invariant hyperplane $H$ through $0$ with 
	unit normal vector $a$ is equivalent to one of the following situations:
	\begin{enumerate}
		\item[(i)] The vector $a$ belongs to $\cone\{e_p,-e_q\}\cap\cone\{u_p,-u_q\}$,
			for some $p,q\in\{1,\dots,m\}$, $p\ne q$.
		\item[(ii)] The inequality $\lng e_p,e_i\rng\le0$ holds for some 
			$p\in\{1,\dots,m\}$ and any 
			$i\in\{1,\dots,m\}$ with $i\ne p$, and $a=\pm e_p/\|e_p\|$.
		\item[(iii)] The inequality $\lng u_p,u_j\rng\le0$ holds for some 
			$p\in\{1,\dots,m\}$ and any 
			$j\in\{1,\dots,m\}$ with $j\ne p$, and $a=\pm u_p/\|u_p\|$.
	\end{enumerate}
\end{corollary}
\begin{proof}
	$\,$
	\medskip

	Let us see first that the conditions are sufficient.
	
	By Lemma \ref{foo}, $H$ is a $K$-invariant hyperplane through 
	$0$ with unit normal vector $a$ if and only if
	\begin{equation}\label{bineg}
		\alpha^i\beta^j\le\delta_{ij},
	\end{equation} 
	for any $i,j\in\{1,\dots,m\}$, where 
	\begin{equation}\label{enormal}
		a=\alpha^1e_1+\dots+\alpha^me_m=\beta^1u_1+\dots+\beta^mu_m. 
	\end{equation}

	Assume that item (i) holds. Then, $a=\alpha^pe_p+\alpha^qe_q=\beta^pu_p+\beta^qu_q$,
	where $\alpha^p\ge0$, $\alpha^q\le0$, $\beta^p\ge0$ and $\beta^q\le0$. Hence, we have
	$\alpha^p\beta^q=\alpha^q\beta^p\le0$. On the other hand, 
	$1=\|a\|^2=\lng\alpha^pe_p+\alpha^qe_q,\beta^pu_p+\beta^qu_q\rng=\alpha^p\beta^p+
	\alpha^q\beta^q$, $\alpha^p\beta^p\ge0$ and $\alpha^q\beta^q\ge0$ imply 
	$\alpha^p\beta^p\le1$ and $\alpha^q\beta^q\le1$. Hence, conditions \eqref{bineg} are
	satisfied.
	\medskip

	Assume that item (ii) holds. Then, $\alpha^p=\pm1/\|e_p\|$ and all other 
	$\alpha^i$s are $0$. We also have 
	$\beta^j=\lng a,e_j\rng=\pm(1/\|e_p\|)\lng e_p,e_j\rng$. Thus, 
	$\alpha^i\beta^j=0\le\delta_{ij}$ if $i\ne p$ and $\alpha^p\beta^j=
	(1/\|e_p\|)^2\lng e_p,e_j\rng\le\delta_{jp}$. Hence, conditions \eqref{bineg} are
	satisfied.
	\medskip

	Assume that item (iii) holds. An argument similar to the one in the previous paragraph 
	shows that conditions \eqref{bineg} are satisfied.
	\medskip
	
	To see that the conditions are necessary, first suppose that more than 
	two $\beta^j$s are nonzero. Then, there exists two 
	$\beta^j$s which have the same sign. Without loss of generality we can suppose that 
	they are both positive. Then, from inequalities \eqref{bineg}, it follows that all 
	$\alpha^i$s are nonpositive. Thus, $a\in -K$. If all $\beta^j$s are nonnegative, then 
	$a\in K^*$. Hence, $a\in K^*\cap(-K)=\{0\}$, which is a contradiction. Therefore, there
	exists a $p\in\{1,\dots,m\}$ such that $\beta^p<0$. But then, by inequalities 
	\eqref{bineg}, all $\alpha^i$s with $i\ne p$ are nonnegative. Hence, all $\alpha^i$s,
	with $i\ne p$ are zero. Hence, from equation \eqref{enormal} we get $a=\alpha^pe_p$ and
	therefore $a=\pm e_p/\|e_p\|$. The same equation implies
	\begin{equation}\label{enormal2}
		\alpha^pe_p=\beta^1u_1+\dots+\beta^mu_m. 
	\end{equation}
	Take any $i\in\{1,\dots,m\}$ with $i\ne p$. 
	Since $a=\alpha^pe_p$ and $\alpha^k\leq 0$ for all $k\in\{1,\dots,m\}$, it follows 
	that $\alpha^p<0$. Hence, inequality \eqref{bineg} implies $\beta^i\ge0$. 
	Thus, by multiplying equation \eqref{enormal2} scalarly by $e_i$ and by using
	relations \eqref{biort}, we obtain
	\begin{equation}\label{enormal3}
		\alpha^p\lng e_p,e_i\rng=\beta^i\ge0. 
	\end{equation}
	Now, if we divide \eqref{enormal3} by $\alpha^p<0$, then we obtain 
	$\lng e_p,e_i\rng\le0$. The latter inequality together with $a=\pm e_p/\|e_p\|$ implies 
	that item (ii) holds.
	
	Next, suppose that at most two $\beta^j$s are nonzero. 	
	\begin{enumerate}
		\item[]
			\begin{enumerate} 
				\item[Case 1.] Suppose that there exists $p,q\in\{1,\dots,m\}$
					with $p\ne q$ such that $\beta^p\beta^q>0$. Then, 
					without loss of generality we can suppose that 
					$\beta^p>0$ and $\beta^q>0$. Since all $\beta^j$'s are
					nonnegative and there exists two $\beta^j$'s which are
					positive, as above, we can show that this would imply 
					that $a=0$ which is a contradiction. Thus, this case 
					cannot hold.
				\item[Case 2.] Suppose that there exists $p,q\in\{1,\dots,m\}$ 
					with $p\ne q$ such that $\beta^p>0$ and $\beta^q<0$. 
					Then, by using inequalities \eqref{bineg}, we must have
					$\alpha^i=0$, for all 
					$i\in\{1,\dots,m\}\setminus\{p,q\}$.
					By using again inequalities \eqref{bineg}, we obtain
					$\alpha^p\ge0$ and $\alpha^q\le0$. It follows that
					$a=\alpha^pe_p+\alpha^qe_q=\beta^pu_p+\beta^qu_q$. 
					Thus, $\cone\{e_p,-e_q\}\cap\cone\{u_p,-u_q\}\ne\{0\}$ 
					and  $a\in\cone\{e_p,-e_q\}\cap\cone\{u_p,-u_q\}$. So, 
					in this case item (i) holds. 
				\item[Case 3.] Suppose that only one of the $\beta^j$ 
					is nonzero, that is $a=\beta^pu_p$. Since $\|a\|=1$,
					it follows that $a=\pm u_p/\|u_p\|$. An argument 
					similar to the one following relation 
					\eqref{enormal2} shows that 
					$\lng u_p,u_j\rng\le0$, for any 
					$j\in\{1,\dots,m\}$, $j\ne p$. So, in this case item 
					(iii) holds.
			\end{enumerate}
	\end{enumerate}
\end{proof}

Now, combining the results in Lemma \ref{latprop}, Corollary \ref{POLYH} and Lemma \ref{foo},
we have the following result:

\begin{corollary}\label{HALOINV} 
Let $K\subset\R^m$ be a simplicial cone given by (\ref{youdine}) such that its dual $K^*$
	is of form (\ref{youdinedual}) with the vectors $u_1,\dots,u_m$ satisfying (\ref{biort}).
	Let $C$ be a closed convex set with nonempty interior.
	Then, the following assertions are equivalent:
	 \begin{enumerate}
	\item [\emph{(i)}] $C$ is a $K$-invariant set, i. e., it is
	invariant with respect to the operations $\sa$, $\su$, $\sa_*,$ and $\su_*$ ;
	\item [\emph{(ii)}] The projection $P_C$ is $K$-isotone and $K^*$-isotone;
	\item [\emph{(iii)}] The set $C$ can be represented by
	\begin{equation}\label{redu2}
	C=\cap_{i\in \N} H_-(a_i,b_i),
	\end{equation}
	where each hyperplane $H(a_i,b_i),\;i\in \N$ is tangent to $C$ and is $K$-invariant
	and $a_i$ are unit normals with 
	\begin{equation}\label{haloeset}
	a_i=\alpha_i^1e_1+\dots+\alpha_i^me_m=\beta_i^1u_1+\dots+
		\beta_i^mu_m,\; \alpha_i^k\beta_i^l\leq \delta_{kl},\;i\in \N;
		\end{equation}
	\item [\emph{(iv)}] $C$ can be represented by (\ref{redu2}),
	where each hyperplane $H(a_i,b_i),\;i\in \N$ with the unit normals $a_i, \;i\in \N$
	satisfying the conditions (\ref{haloeset}), is tangent to $C$ and is a 
	$K$-isotone projection and a $K^*$-isotone
	projection set.
	\end{enumerate}
	\end{corollary}


	In the particular case of $K=\R^m_+$, where $\R^m_+$ is the positve
	orthant of a Cartesian system, taking into account that
	$\sa=\sa_*=\wedge,\;\su=\su_*=\vee,$ the Corollary \ref{HALOINV},
	(via Corollary \ref{orthalo} in place of Corollary \ref{cfoo})
	takes the form:

	\begin{corollary}

	Let $C$ be a closed convex set with nonempty interior of the coordinatewise ordered
	Euclidean space $\R^m$. Then, the following assertions are equivalent:
	\begin{enumerate}
	\item [\emph{(i)}] The set $C$ is a sublattice;
	\item [\emph{(ii)}]The projection $P_C$ is isotone;
	\item [\emph{(iii)}]\begin{equation*}
	C=\cap_{i\in\N} H_-(a_i,b_i),
	\end{equation*}
	where each hyperplane $H(a_i,b_i)$ is tangent to $C$ and the normals $a_i$ 
	are nonzero vectors $a_i=(a_i^1,...,a_i^m)$ with the properties 
		$a_i^ka_i^l\leq 0$
		whenever $k\not= l,\;\;i\in \N.$

\end{enumerate}
\end{corollary}

This corollary is exactly Corollary 4 in \cite{NemethNemeth2012a}.

\begin{example}
	$\,$

	\begin{enumerate}
		\item Let $\oplus$ denote orthogonal direct sum of subspaces, $m=2k$ and 
			$\R^m=V_1\oplus\dots\oplus V_k$, where $V_1,\dots,V_k$ are pairwise 
			orthogonal two-dimensional subspaces of $\R^m$. Let 
			$\{e_{2i-1},e_{2i}\}$ be a basis of the subspace $V_i$, for any 
			$i\in\{1,\dots,k\}$. Then, $K:=\cone\{e_1,\dots,e_m\}$ is a simplicial 
			cone and let $K^*=\{u_1,\dots,u_m\}$ its dual cone. By using the 
			biorthogonality of the vectors $e_i$ and $u_j$, $i,j\in\{1,\dots,m\}$,
			we obtain that 
			$\cone\{e_{2i-1},-e_{2i}\}\cap\cone\{u_{2i-1},-u_{2i}\}=
			\cone\{u_{2i-1},-u_{2i}\}\neq\{0\}$ if 
			$\lng e_{2i-1},e_{2i}\rng\ge0$ and 
			$\cone\{e_{2i-1},-e_{2i}\}\cap\cone\{u_{2i-1},-u_{2i}\}=
			\cone\{e_{2i-1},-e_{2i}\}\neq\{0\}$ if 
			$\lng e_{2i-1},e_{2i}\rng\le0$. Hence, by using item (i) of Corollary
			\ref{cfoo}, any hyperplane $H$ through $0$ with normal unit vector in 
			$\cone\{e_{2i-1},-e_{2i}\}\cap\cone\{u_{2i-1},-u_{2i}\}\neq\{0\}$ is 
			$K$-invariant.
		\item Let $K$ be an isotone projection cone. Then, by item (ii) of Corollary
			\ref{cfoo}, the $n-1$ dimensional hyperfaces of $K$ are $K$-invariant.
	\end{enumerate}
\end{example}

%

\section{Comments and open questions}\label{concl}

Motivated by isotone type iterative methods for variational inequalities in this paper we 
considered the following very general question: Which are the closed convex 
sets which possess a projection onto them which is isotone with respect to an order relation
defined by a given cone? We showed that they are exactly the sets which are invariant with
respect to some extended lattice operations defined by the projection onto the cone and 
presented some examples for self-dual and simplicial cones. A related at least as interesting 
and still open question is: Which are the closed convex sets for which there exist a cone such 
that the projection onto them are isotone with respect to order relation defined by the cone? 
Both of the above questions can also be formulated for a general pair of cones and their 
corresponding orderings. These extended questions are also open.

We remark that several of our results do not use explicitly that the ambient space is
Euclidean, and they hold in Hilbert spaces too.%



\bibliographystyle{habbrv}
\bibliography{halosz}

\begin{thebibliography}{10}

\bibitem{CarlHeikilla2011}
S.~Carl and S.~Heikkil\"a.
\newblock {\em Fixed point theory in ordered sets and applications}.
\newblock Springer, New York, 2011.
\newblock From differential and integral equations to game theory.

\bibitem{GowdaSznajderTao2004}
M.~S. Gowda, R.~Sznajder, and J.~Tao.
\newblock Some p-properties for linear transformations on {Euclidean} {Jordan}
  algebras.
\newblock {\em Linear Algebra Appl.}, 393:203--232, 2004.

\bibitem{Isac1996}
G.~Isac.
\newblock On the order monotonicity of the metric projection operator.
\newblock In {\em Approximation theory, wavelets and applications ({M}aratea,
  1994)}, volume 454 of {\em NATO Adv. Sci. Inst. Ser. C Math. Phys. Sci.},
  pages 365--379. Kluwer Acad. Publ., Dordrecht, 1995.

\bibitem{IsacNemeth1986}
G.~Isac and A.~B. N\'emeth.
\newblock Monotonicity of metric projections onto positive cones of ordered
  {E}uclidean spaces.
\newblock {\em Arch. Math.}, 46(6):568--576, 1986.

\bibitem{IsacNemeth1990b}
G.~Isac and A.~B. N\'emeth.
\newblock Isotone projection cones in {H}ilbert spaces and the complementarity
  problem.
\newblock {\em Boll. Un. Mat. Ital. B.}, 7(4):773--802, 1990.

\bibitem{IsacNemeth2008c}
G.~Isac and S.~Z. N\'emeth.
\newblock Regular exceptional family of elements with respect to isotone
  projection cones in {Hilbert} spaces and complementarity problems.
\newblock {\em Optim Lett.}, 2(3):567--576, 2008.

\bibitem{Krasnoselskii1962}
M.~A. Krasnosel'skii.
\newblock {\em Positive Solution of Operatorial Equations}.
\newblock Fizmatgiz, Moskow, 1962.

\bibitem{Moreau1962}
J.~J. Moreau.
\newblock D\'ecomposition orthogonale d'un espace hilbertien selon deux c\^ones
  mutuellement polaires.
\newblock {\em C. R. Acad. Sci.}, 255:238--240, 1962.

\bibitem{NemethNemeth2009}
A.~B. N\'emeth and S.~Z. N\'emeth.
\newblock How to project onto an isotone projection cone.
\newblock {\em Linear Algebra Appl.}, 433(1):41--51, 2010.

\bibitem{NemethNemeth2012}
A.~B. N\'emeth and S.~Z. N\'emeth.
\newblock A duality between the metric projection onto a cone and the metric
  projection onto its dual.
\newblock {\em J. Math. Anal. Appl.}, 392(2):172--178, 2012.

\bibitem{NemethNemeth2012a}
A.~B. N\'emeth and S.~Z. N\'emeth.
\newblock Selfdual cones, generalized lattice operations and isotone
  projections, 2012, arXiv:1210.2324.

\bibitem{Nemeth2009a}
S.~Z. N\'emeth.
\newblock Iterative methods for nonlinear complementarity problems on isotone
  projection cones.
\newblock {\em J. Math. Anal. Appl.}, 350(1):340--347, 2009.

\bibitem{NishimuraOk2012}
H.~Nishimura and E.~A. Ok.
\newblock Solvability of variational inequalities on {Hilbert} lattices.
\newblock {\em Preprint}, pages 1--28, 2012.

\bibitem{QueyranneTardella2006}
M.~Queyranne and F.~Tardella.
\newblock Bimonotone linear inequalities and sublattices of $\mathbb{R}^n$.
\newblock {\em Linear Algebra Appl.}, 413:100--120, 2006.

\bibitem{Rockafellar1970}
R.~T. Rockafellar.
\newblock {\em Convex Analysis}.
\newblock Princeton: Princeton Univ. Press, 1970.

\bibitem{Topkis1976}
D.~M. Topkis.
\newblock The structure of sublattices of the product of n lattices.
\newblock {\em Pacific J. Math.}, 65:525--532, 1976.

\bibitem{Veinott1981}
A.~F. Veinott.
\newblock Representation of general and polyhedral sublattices and sublattices
  of product spaces.
\newblock {\em Linear Algebra Appl.}, 114/115:172--178, 1981.

\bibitem{Youdine1939}
A.~Youdine.
\newblock Solution des deux probl\`emes de la the\'eorie des espaces
  semi-ordonn\'es.
\newblock {\em Comptes Rendus (Doklady) de l'Acad\'emie des Sciences de
  l'URSS}, 23:418--422, 1939.

\bibitem{Zarantonello1971}
E.~Zarantonello.
\newblock Projections on convex sets in {Hilbert} space and spectral theory,
  {I}: {P}rojections on convex sets, {II}: {S}pectral theory.
\newblock {\em Contrib. Nonlin. Functional Analysis, Proc. Sympos. Univ.
  Wisconsin, Madison}, pages 237--424, 1971.

\end{thebibliography}

\end{document}